\documentclass[11pt,twoside,leqno]{article}
\usepackage[utf8]{inputenc}
\usepackage[top=1.9in,bottom=1.in, left=1.7in,right=1.7in, includefoot]{geometry}
\usepackage{fancyhdr}
\usepackage{varioref}
\usepackage{enumerate}
\usepackage{graphicx}
\usepackage{wrapfig}
\usepackage[font=footnotesize,labelfont=bf]{caption}
\usepackage{float}
\usepackage{tabularx}
\usepackage{ragged2e}
\usepackage{enumitem,kantlipsum}
\usepackage{changepage}   % for the adjustwidth environment of paragraphs
\usepackage{setspace} 
    \fancypagestyle{plain}{
    \fancyhf{}}
\usepackage{titlesec}
\titleformat{\section}[hang]{\bf\centering}{\thesection.}{1em}{}
\titleformat{\subsection}[hang]{\slshape}{\thesubsection.}{1em}{}
\titleformat{\subsubsection}[hang]{\slshape}{\thesubsubsection.}{1em}{}
\usepackage{bm}
\usepackage{type1cm}
\usepackage{lettrine}
\usepackage{amsfonts}
\usepackage{amssymb}
\usepackage{amsmath}
\usepackage{graphicx}
\usepackage{array,multirow,graphicx}
\usepackage{float}
\usepackage{amsthm}
\usepackage{mathrsfs}
\usepackage{pifont} 
\usepackage{bm}
\usepackage[makeroom]{cancel}
\usepackage{rotating}
\usepackage{longtable}
\newtheoremstyle{theoremdd}% name of the style to be used
  {\topsep}% measure of space to leave above the theorem. E.g.: 3pt
  {4pt}% measure of space to leave below the theorem. E.g.: 3pt
  {\itshape}% name of font to use in the body of the theorem
  {0pt}% measure of space to indent
  {\scshape}% name of head font
  {.}% punctuation between head and body
  { }% space after theorem head; " " = normal interword space
  {\thmname{#1}\thmnumber{ #2}\textnormal{\thmnote{ (#3)}}}
\theoremstyle{theoremdd}

\newtheorem{Teo}{Theorem}[section]
\newtheorem{Prop}[Teo]{Proposition}
\newtheorem{Lema}[Teo]{Lemma}
\newtheorem{Coro}[Teo]{Corollary}
\newtheorem{Def}[Teo]{Definition}

\newtheorem{Rule}{Rule}

\usepackage[dvipsnames]{xcolor}
\numberwithin{equation}{section} % Para enumera ecuaciones segun capitulo
\usepackage{tocloft,lipsum,pgffor}
\makeatletter
\def\l@chapter#1#2{\pagebreak[3]
   \vskip 1.0em plus 1pt  % space above chapter line
   \@tempdima 1.5em       % width of box holding chapter number
   \begingroup
     \parindent \z@ \rightskip \@pnumwidth
     \parfillskip -\@pnumwidth
     \bfseries            % Boldface removed. 2006nm %% <= \bfseries uncommented again
     \leavevmode          % TeX command to enter horizontal mode.
     #1\hfil \hbox to\@pnumwidth{\hss #2}\par
   \endgroup}
\def\l@chapter{\@dottedtocline{0}{0em}{2.3em}}  %added this line 2006nm %Changed 09/05/2012
\def\l@section{\@dottedtocline{1}{1.5em}{2.3em}} %changed 1 to 0 2006nm
\def\l@subsection{\@dottedtocline{2}{3.8em}{3.2em}}
\def\l@subsubsection{\@dottedtocline{3}{7.0em}{4.1em}}
\def\l@paragraph{\@dottedtocline{4}{10em}{5em}}
\def\l@subparagraph{\@dottedtocline{5}{12em}{6em}}
\makeatother

\let\OLDthebibliography\thebibliography
\renewcommand\thebibliography[1]{
  \OLDthebibliography{#1}
  \setlength{\parskip}{0pt}
  \setlength{\itemsep}{0pt plus 0.3ex}
 }

\newcolumntype{R}[1]{>{\RaggedRight}p{#1}}

\begin{document}
\sloppy 
%
%---------------------------------------------
%   TITLE & FOOTNOTES
%---------------------------------------------
%
\title{\Large\bf{The Cartesian method and Fermat's Last Theorem}}
\author{\scshape{\normalsize{By Hector Ivan Nunez}}}
\date {}
\maketitle
\let\thefootnote\relax\footnotetext{\textit{\today.}}
\let\thefootnote\relax
\let\thefootnote\relax
%
%---------------------------------------------
%ABSTRACT
%---------------------------------------------
%
\begin{abstract}
Fermat's Last Theorem is proved by using the philosophical and mathematical knowledge of 1637 when the French mathematician \textit{Pierre de Fermat} claimed to have a truly marvelous proof of his conjecture. Our approach consists of setting three variables of Fermat's equation as integers and then evaluating whether the remaining variable can be an integer as well. Pythagorean triples play a fundamental role in claiming that at least an irrational number is needed to satisfy Fermat's equation. As a result, we confirm that Fermat's Last Theorem is valid.
\end{abstract}
%
%--------------------------------------------------
%\section{Contents}
%--------------------------------------------------
%
%\tableofcontents
%\addtocontents{toc}{~\hfill\textbf{Page}\par}
%
%----------------------------------------------
\section{Introduction}
%----------------------------------------------
%
While \textit{Sir Andrew Wiles} already proved Fermat’s Last Theorem in 1995 \cite{wiles}, there is a gap in scientific knowledge about proving it using the simple mathematical machinery available when \textit{Pierre de Fermat} conjectured it in 1637. This paper aims to present meaningful progress toward filling that gap.

We consider the philosophical method of analysis known today as the \textit {Cartesian method} \cite{MC}, which was coincidentally proposed in the latter mentioned year, 1637, by the French philosopher and mathematician \textit{René Descartes}.  The four rules of this method are described in Subsection~\ref{metodo}. The Cartesian method divides a complex problem into as many parts as possible until its fundamental composition is found. Then, we have to solve those simple parts that appear. Afterward, group them, and synthesize the problem. 

It is well-known that Fermat's Last Theorem, also referred to as the \textit{Fermat-Wiles Theorem}, has a mathematical structure remarkably similar to the Pythagorean Theorem. One may ask whether Fermat's claim can be studied using the Cartesian method with analogies to the Pythagorean theorem. In fact, that is what our exposition is about, and we give an affirmative answer to that question.  

An exciting account of the history of this famous theorem can be found in books \cite {SS} and articles. Therefore, we just say its statement:

\begin{quotation}
\noindent \textit{If $n$ is an integer greater than $2$, then there are no positive integers \mbox{$a,b,c$} that satisfy the equation $a^n+b^n=c^n$.}
\end{quotation} 

To prove Fermat's Last Theorem, we begin by developing a geometric representation of the equation \mbox{$a^n+b^n=c^n$}. Next, we study the arithmetic structure of that equation whenever $a,b,c$ and $n$ are positive integers with $n>2$. 

Then, we divide the problem into two parts: when $n$ is even and when $n$ is odd, with $n>2$ and $a, b, c$ as positive integers.

When $n$ is even, we use the Fundamental Theorem of Arithmetic, and again we divide the problem into more simple parts. Using analogies to the Pythagorean theorem, we investigate whether a simplified form of the equation $a^n+b^n=c^n$ is valid in these parts. The concept of Pythagorean triples makes a final invitation to divide the problem into several more parts once again. The Pythagorean triples point out that such an equation is not valid since there are no three integers but at least an irrational number in the triple $(a, b, c)$.

Next, to solve the remaining case when $n$ is odd and greater than~$2$, we consider perfect square numbers. Square numbers link to the case when $(a,b,c)$ are positive integers, and $n$ is odd. By contradiction, we find no three integers in the triple $(a, b, c)$ but at least an irrational number. Finally, when we synthesize the cases for $n$ an even integer and $n$ an odd integer, it results in a proof of Fermat's Last Theorem.

Section~\ref{final} shows that if $(a,b,c)$ are positive integers, then $n$ is irrational. As well, we prove that $a^n+b^n=c^n$ has no solutions with fractional numbers.
\begin{Def}
Define $a^n+b^n=c^n$ a Fermat's equation. In particular, a primitive Fermat's equation is when $gcd(a,b,c)=1$. 
\end{Def}
\noindent \textit{Remark.} We will analyze primitive Fermat's equations only. It suffices for our proof since non-primitive Fermat's equations are simply a multiple of those. 

Herein, the universe of discourse has a partition of three disjoint sets, namely $\mathbb{Z}^+$, \mbox{$\mathbb{Q}^*$} and $\mathbb{Q}^c$, such that $\mathbb{Z}^+ \cap \mathbb{Q}^*=\emptyset$,  $\mathbb{Z}^+ \cap \mathbb{Q}^c=\emptyset$, $\mathbb{Q}^* \cap \mathbb{Q}^c=\emptyset$ and \mbox{$\mathbb{Z}^+ \cup \mathbb{Q}^* \cup \mathbb{Q}^c=\mathbb{R}^+$}. The next section clears up this notation.
%
%----------------------------------------------
\section{Notation}
%----------------------------------------------
%
\begin{longtable}{p{0.12\textwidth}R{0.78\textwidth}}
\endhead
$a,b,c,n$ & Arbitrary positive integers with $b\leq a<c$, and $n>2$.\\
$(a,b, c)$ & A triple.\\
$gcd(a,b,c)$ & The greatest common divisor of the integers $a,b$ and $c$.\\
$d,s$& Positive odd integers.\\
$k,p,q,r$& Positive integers.\\
$\mathbb{Z}^+$& The set of positive integers.\\
$\mathbb{N}$& The set of positive integers and $0$.\\
$\mathbb{Q}^+$ & The subset of positive rational numbers.\\
$\mathbb{Q}^*$ & The subset of positive rational numbers with no integers. Then its elements are $p/q$ such that $q\neq 0$, $q\neq 1$ and \mbox{$gcd(p,q)=1$.}\\
$\mathbb{Q}^c$ & Positive irrational numbers with no integers or rationals.\\
$\mathbb{R}^+$& The set of positive real numbers.\\
$ \Delta ABC$ &  A triangle with $A,B$ and $C$ vertexes.\\
$x,y,h$& Real numbers.\\
$Even$ & Positive even integers.\\
$Odd$ & Positive odd integers.
\end{longtable}
%
%----------------------------------------------
\section{Mathematical structure of Fermat's Last Theorem}
%----------------------------------------------
%
\subsection{Geometric representation.}
\label{gr}
%----------------------------------------------
%
Section~\ref{gr} is devoted to developing a geometric representation of the equation \mbox{$a^n+b^n=c^n$} in the space of real numbers. That will help us understand the dynamics of the four variables $a,b,c,n$ from a geometrical point of view, especially when they are set as integers. This geometry will be referenced in Lemma~\ref{bnodividea} and Section~\ref{final}.

The equation \mbox{$a^n+b^n=c^n$} written as a function of $n$ is $c(n)=\sqrt[n]{a^n+b^n}$. We can graph this equation in the $(x,y)$ Cartesian plane, where $x,y \in \mathbb{R}^+$, $n$ is the independent variable, and $c$ is the dependent variable.

Let us draw the geometric representation of the equation $c(n)=\sqrt[n]{a^n+b^n}$ in a Pythagorean way, that is, like a triangle. In the first quadrant of the Cartesian plane, we draw $\Delta ABC$ with one side of length $a$ and another of length $b$. The third side is set to a length of $c=\sqrt[n]{a^n+b^n}$. Side $a$ extends from point $(0,0)$ to $(a, 0)$, while side $b$ has one end fixed at $(a,0) $.

\newpage
Let vertex $C$ be able to freely rotate according to the circle drawn by the side~$b$ from point $(a,0)$. Notice that the location of vertex $C$ depends on the value of $n$. 

Because $b\leq a<c $, the vertex $C$ of $ \Delta ABC$ cannot come into contact with the circle drawn by the side~$a$ from $(0,0)$. As~$n$ rises in value, vertex~$C$ approaches this circular area. If the value of~$n$ decreases, that is, $n\rightarrow 1$, then vertex~$C$ rotates to the right, approaching the $x$-axis. Following is an illustration of these details.

\begin{figure}[ht]
  \centering
  \caption{Fermat's Last Theorem graphed as a more general representation than Pythagoras' theorem.} 
  \label{fig:Fig1}
   \includegraphics[width=.98\textwidth]{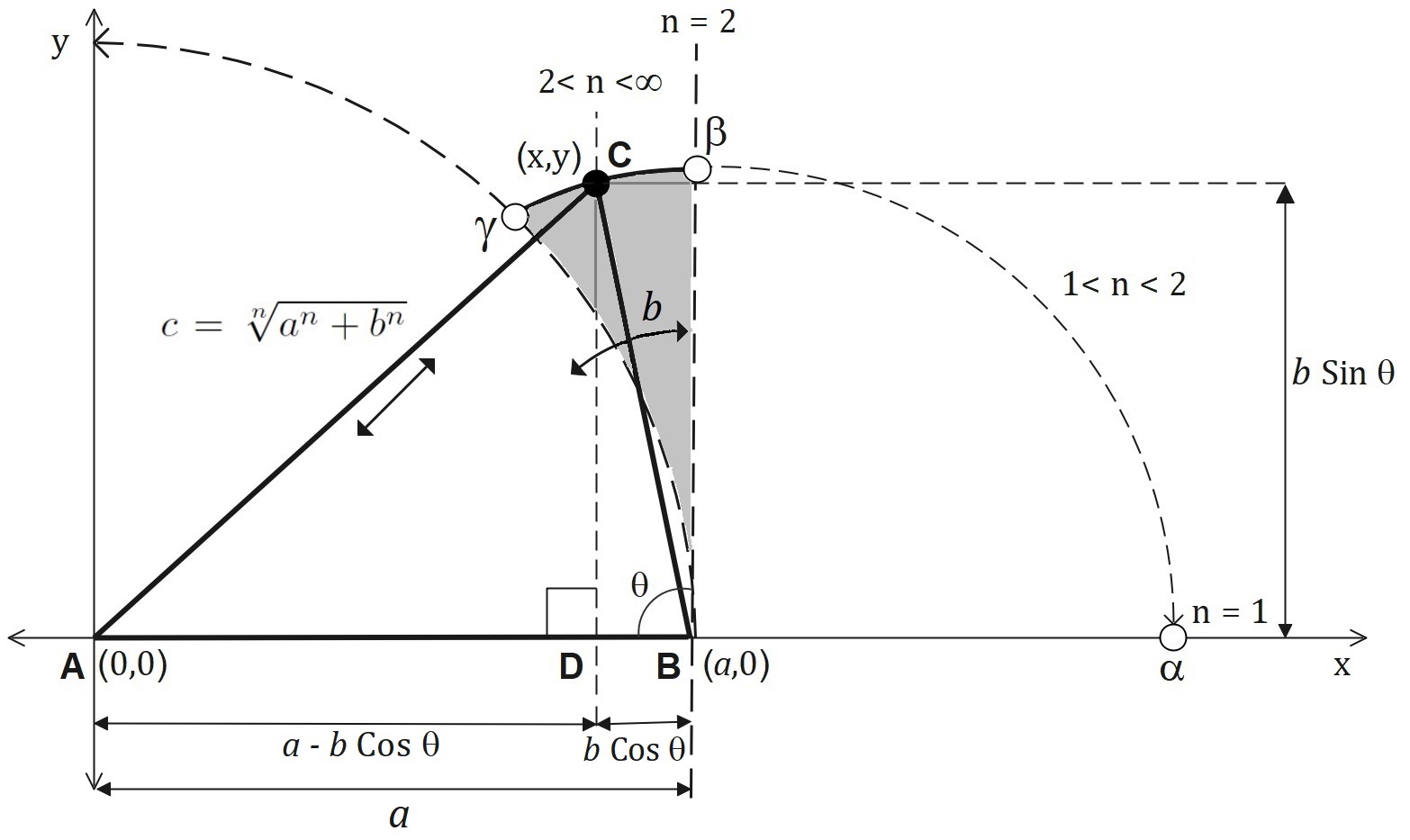}
\end{figure}

Let $S$ be the space formed by the points $(x,y)$ such that $a,b,c,n\in \mathbb{R}^+$ in the equation $c^n = a^n+b^n$ with $n>2$. That space is shown in the shaded area of Figure~\ref{fig:Fig1}, and is bounded between the \mbox{arc $\beta \gamma$} inclusive, and it is when $a=b$; \mbox {arc $\gamma B$} not inclusive because $c>a$, and line $B\beta $ not inclusive because $n>2$.
 
Certain relevant characteristics of Fermat's Last Theorem appear in this area. Suppose that~$a$ and~$b$ are positive integers. Then, there are~$b$ arcs $\beta \gamma $, and for each of them, there could be$^\dagger$ certain points for the vertex~$C$ such that~$c$ is also a positive integer. In Section~\ref{app}, we prove that~$n$ and~$\theta$ are irrationals at that point. 

Consider Figure~\ref{fig:Fig1}. Given that $a,b \in \mathbb{Z}^+$ and  $c, n \in \mathbb{R}^+$, then the minimum~$c$ occurs when~$c \rightarrow a$ and $n \rightarrow \infty$.  The maximum~$c$ in the arc $\beta \gamma $ occurs when $a=b$ and $n \rightarrow 2$, hence $a<c<a \sqrt{2}$.

Observe that $\Delta ABC$ tends to be an equilateral triangle whenever $a=b$ and $n\rightarrow \infty$. If $a>b$ and $n \rightarrow \infty$, then $\Delta ABC$ tends to be an isosceles triangle. $\Delta ABC$ is obtuse if $1<n<2$; otherwise, $\Delta ABC$ is acute if $n>2 $. At $n=2$, $\Delta ABC$ is a right triangle.
%
%----------------------------------------------
\subsection{Arithmetical structure}
%----------------------------------------------
%
Here, we develop some straightforward but valuable concepts that are so often needed to apply contradiction techniques \cite{HTP} that they deserve singling out.

\begin{Lema}
\label{abc_paresimpares}
Assume that equation $a^n+b^n=c^n$ holds for $a,b,c,n$ positive integers and $n>2$. Then a single element of the triple $(a,b,c)$ is even. 
\end{Lema}
\begin{proof}
Let us prove it by contradiction. Any odd raised to a positive integer power is odd, and any even raised to a positive integer power is even. The sum of two odd numbers is even, and the sum of an odd number with an even number is odd. Therefore, it is impossible that in the equation $c^n=a^n+b^n$, we have a triple $(a,b,c)$ with all its elements odd because otherwise, we would have an odd number equal to an even one, which is absurd.

Suppose that we have two even and one odd in the triple $(a,b,c)$. If the odd is~$c$, then the sum $a^n+b^n$ is even, then the equality is absurd. If~$c$ is an even number, the sum of an odd and an even number in $a^n+b^n$ is odd; then, the equality is absurd. Therefore, in the triple $(a,b,c)$, there do not exist two even and one odd.

The case when the three elements of the triple $(a, b,c)$ are even is not considered because we stated in the Introduction (Assumptions) that the equation $c^n=a^n+b^n$ is simplified, then $a,b,c$ have no common factors. 
\end{proof}
\begin{Coro}
\label{trio2kd}
Assume that equation $a^n+b^n=c^n$ holds for $a,b,c,n$ positive integers and $n>2$. Applying the Fundamental Theorem of Arithmetic follows that this equation has one of the following three forms:
\begin{align}
\label{2kD=a+b}
(2^{k}d)^n&=a^n+b^n && \text{a, b, d are odd and $k\in\mathbb{Z}^+$},\\
\label{Ec2}
c^n&=a^n+(2^{k}d)^n && \text{a, c, d are odd and $k\in\mathbb{Z}^+$},\\
\label{Ec3}
c^n&=(2^{k}d)^n+b^n && \text{b, c, d are odd and $k\in\mathbb{Z}^+$}.
\end{align}
\end{Coro}
\textit{Remark.} Notice that \eqref{Ec2} and \eqref{Ec3} look-alike from the commutative property of addition point of view. Nevertheless, we have stated that $b\leq a<c$. Then, the three equations will be evaluated individually by cases to keep the generality and formality.

\begin{Lema}
\label{cnoesQ}
Suppose $a^n+b^n=c^n$ has $a,b,n$  as positive integers with $n>2$. Then, $c$ is not a fractional number.  Furthermore, $c\in \mathbb{Z}^{+}$ or $c \in \mathbb{Q}^{c}$.  
\end{Lema}
\begin{proof}
Let us prove it by contradiction.  Assume that $c \in \mathbb{Q}^*$ such that $c=p/q$ with $q\not\in \{0,1\}$. Then \mbox{$c^n=a^n+b^n$} can be written as $(p/q)^n=a^n+b^n$. The right-hand side of the equal sign is a positive integer since $a,b,n\in \mathbb{Z}^+$; however, the left-hand side is a fractional number since $(p/q)^n$ is irreducible. Therefore, the equation is absurd.
\end{proof}

\begin{Coro}
\label{noracionales}
Suppose that two elements in $(a,b,c)$ are positive integers. Then, to satisfy $a^n+b^n=c^n$ with $n>2$, the third element is not a fractional number.  
\end{Coro}

\begin{Lema}
\label{bnodividea}
Let $b\leq a<c$ with $b\mid a$. Then there are no positive integers \mbox{$a,b,c$} that satisfy  $a^n+b^n=c^n$ with $n>2$. In particular, Fermat's Last Theorem is valid.
\end{Lema}

\begin{proof}
If $b\mid a$, then \mbox{$a/b\in\mathbb{Z}^+$}.
Rewriting \eqref{2kD=a+b} gives $$\frac{2^kd}{b}=\sqrt[n]{(a/b)^n+1}.$$ 

If~$n=2$, then $(2^kd)/b\in\mathbb{Q}^c$ since there is no Pythagorean triple that has a number~$1$ in the triple. In fact, $a/b<\sqrt{(a/b)^2+1}<a/b+1$.
    
Let us generalize$^\dagger$ this to $n>2$. From Figure~\ref{fig:Fig1}, we see that if $n$ increases, then $c=2^kd$ decreases, thus $(2^kd)/b$ decreases. Clearly, it is deduced that  $a/b<\sqrt[n]{(a/b)^n+1}<a/b+1$. Substituting gives \mbox{$a/b<(2^kd)/b<a/b+1$}. 

Awfully well, the number $(2^kd)/b$ is \textit{locked down} between two continuous integers. Considering Corollary~\ref{cnoesQ} follows that $2^kd\in\mathbb{Q}^c$.
\end{proof}

\begin{Prop}
\label{a=b}
Let $a=b$. Then there are no positive integers \mbox{$a,b,c$} that satisfy $a^n+b^n=c^n$ with $n>2$. In particular, Fermat's Last Theorem is valid.
\end{Prop}

\begin{proof}
If $a=b$, then \eqref{2kD=a+b} gives $2^kd=a\sqrt[n]{2}$, resulting in  $2^kd\in\mathbb{Q}^c$.
\end{proof}

\begin{Prop}
Suppose that $a,b,c$ and $n$ are integers such that $n>2$. Then \mbox{$gcd(a,b,c)=1$} if and only if \mbox{$gcd(a^n,b^n,c^n)=1$.}
\end{Prop}
%
%----------------------------------------------
\section{Fermat's Last Theorem proof}
\label{demo}
%
%----------------------------------------------
% n EVEN, n>2
\subsection{n is an even integer greater than 2}
\label{npar}
%----------------------------------------------
%
Following the Cartesian method, let us divide Corollary~\ref{trio2kd}'s equations into two parts, one for $d=1$ and the other for $d>1$. 

%Once we find the fundamental concepts, we will work up to the most difficult. When synthesizing, it will be proved that Fermat's Last Theorem is valid when $n$ is even greater than $2$.
%
%----------------------------------------------
% n EVEN, n>2
\subsubsection{d=1}
%----------------------------------------------

\begin{Lema}
\label{2knpar}
Let $n$ be even greater than 2. If one element in the triple $(a,b,c)$ has the form $2^k$ where $k\in\mathbb{Z}^+$, then there are no positive integers \mbox{$a,b,c$} satisfying the equation $a^n+b^n=c^n$.  In particular, Fermat's Last Theorem is valid.
\end{Lema}

\begin{proof}
 By contradiction, we will prove that $2^{k}\in\mathbb{Q}^c$. According to Lemma~\ref{abc_paresimpares}, in $(a,b,c)$ there exist one even and two odd. Let us assume that the even is $2^{k}$ where $k\in\mathbb{Z}^{+}$. Next, by cases, we  analyze the three equations of Corollary~\ref{trio2kd} with $d=1$.

\textit{Case 1}. Notice that the left-hand side of \eqref{2kD=a+b} with $d=1$ is $2^k$, which is an even integer because $k\neq 0$ for $b\leq a<2^k$. Since $gcd(a,b,2^k) =1$, by Lemma~\ref{abc_paresimpares}, we know that $a$ and $b$ are odd. 

In summarizing, we have a primitive Fermat's equation with $n$ as an even integer greater than $2$, $a$ and $b$ are odd integers, $c=2^{k}$ (even), $k \in \mathbb{Z}^+$, $b \leq a<2^k$, and it has the form
\begin{align}
\label{2kn}
2^{kn}&=a^n+b^n.
\end{align}

Let us do an artifice. Rewriting \eqref{2kn} similar to Pythagoras' theorem gives
\begin{align}
\label{2kpitagoras}
[2^{kn/2}]^2=[a^{n/2}]^2+[b^{n/2}]^2.
\end{align}
\textit{Remark.} Since $n$ is even and $a, b$ are odd, then $a^{n/2}$ and $b^{n/2}$ are odd, and $2^{kn/2}$ is even. Besides, the upper exponent is $2$, and \mbox{$gcd(a^{n/2}, b^{n/2}, 2^{kn/2})=1$.} Clearly, a \textit{necessary condition} for \eqref{2kpitagoras} is that $(a^{n/2}, b^{n/2}, 2^{kn/2})$ must be a \textit{primitive} Pythagorean triple. 

 The Pythagorean triples have been studied since the Babylonians between 1900 BC and \mbox{1600 BC \cite{TP}}. Today, we know that \textit{Euclid's Formula} is a generator of \textit{all} primitive Pythagorean triples but not all the non-primitive \cite{PPT}. 
 
 Recall that $gcd(a,b,2^k)=1$, then we are interested in primitive Pythagorean triples only, which requires that the hypotenuse $2^{kn/2}$ is odd. But it  contradicts the fact that $2^{kn/2}$ is even; therefore, \eqref{2kn} is impossible with integers. Lemma~\ref{cnoesQ} and Corollary~\ref{noracionales} assert that $2^{k}\notin \mathbb{Q}^*$. Then, within $(a,b,2^k)$ there exist at least an irrational number. Since $a$ and $b$ are odd, then $2^{k}\in\mathbb{Q}^{c}$. 

\textit{Case 2}. Suppose \eqref{Ec2} with $d=1$.  Then the equation to analyze is
\begin{align}
\label{bpar}
c^n= a^n+2^{kn}.
\end{align}

Rewriting \eqref{bpar} similar to the Pythagorean theorem gives
\begin{align}
\label{b=2k}
[c^{n/2}]^2=[a^{n/2}]^2+[2^{kn/2}]^2.
\end{align}
\textit{Remark.} Since $n$ is even and $a,c$ are odd, then $c^{n/2}, a^{n/2}$ are odd, and $2^{kn/2}$ is even. Besides, $gcd(a,2^k,c)=gcd(a^{n/2}, 2^{kn/2}, c^{n/2})=1$. Since the upper exponent in \eqref{b=2k} is $2$, a \textit{necessary condition} for \eqref{b=2k} is that  $(a^{n/2}, 2^{kn/2}, c^{n/2})$ must be a \textit{primitive} Pythagorean triple. 

Diophantus may have written \textit{Euclid's Formula} as we know it \mbox{today \cite{TP}}.  It points out that to generate the primitive Pythagorean triple $(a^{n/2}, 2^{kn/2}, c^{n/2})$, the following Euclid's Formula must be satisfied for two relative primes, $p$ and $q$, with $p>q$ and not both odd: 
\begin{align}
\label{p2+q2}
c^{n/2}&=p^2+q^2,\\
\label{p2-q2}
a^{n/2}&=p^2-q^2,\\
\label{2pq}
2^{kn/2}&=2pq.
\end{align}

Solving for $p$ in \eqref{2pq} gives $p=2^{kn/2}/(2q)$. Substituting $p$ in \eqref{p2+q2} gives $c^{n/2}=[2^{kn/2}/(2q)]^2+q^2$ which is simplified as
\begin{align}
\label{meollo}
c^{n/2}&=\frac{2^{kn-2}}{q^2}+q^2.     
\end{align}
First,  consider \eqref{meollo} with $q$ even.
\begin{enumerate}
    \item Suppose that $q^2=2^{kn-2}$. Squaring both sides of \eqref{2pq} gives $2^{kn-2}>q^2$, which is a contradiction. Therefore, it is impossible that $q^2=2^{kn-2}$.
    
    \item Assume that $q$ is even with $q^2\mid 2^{kn-2}$ and $2^{kn-2}\neq q^{2}$. Then  $c^{n/2}$ is even, contradicting the fact that $c^{n/2}$ is odd. Hence, it is impossible.
    
    \item Assume that $q$ is even with $q^2\nmid 2^{kn-2}$. Then $c^{n/2}\in\mathbb{Q}^*$, which contradicts that $c^{n/2}$ is odd. Therefore, it is impossible that $q^2\nmid 2^{kn-2}$.  
 \end{enumerate}
Second, consider \eqref{meollo} with $q$ odd.
\begin{enumerate}
    \item Suppose $q=1$. Solving \eqref{p2-q2} gives $p^{2}=a^{n/2}+1$. Substituting in \eqref{p2+q2} and rewriting gives $c^{n/2}-a^{n/2}=2$. Such equality is impossible since $a$ and $c$ are odd with $a<c$, and the minimum difference between them is $2$, but when raising $a$ and $c$ to an integer power greater than $1$, then the difference is greater than $2$. Since $n/2>1$, then $c^{n/2}-a^{n/2}>2$, which is a contradiction. Hence, $q=1$ is impossible.
    
    \item Assume $q$ is odd and greater than $1$. Then $c^{n/2}\in\mathbb{Q}^*$ since $q^2 \nmid 2^{kn-2}$, which contradicts the fact that $c^{n/2}$ is odd. Therefore, $q$ odd greater than $1$ is impossible.
 \end{enumerate}

Thus, \eqref{bpar} cannot be satisfied with integers. From Corollary~\ref{cnoesQ}, it follows that within $(a,2^k,c)$ there exists at least an irrational number. Since $a,c$ are odd, clearly $2^{k}\in\mathbb{Q}^{c}$.

\textit{Case 3}. Suppose \eqref{Ec3} with $d=1$.  Then the equation to analyze is \mbox{$c^n= 2^{kn}+b^n$}. Similar to Case 2, and using the commutative property of addition, it is concluded by contradiction that $2^{k}\in\mathbb{Q}^{c}$.
\end{proof}
\textit{Remark}. The value of Lemma \ref{2knpar} is that it proves that \mbox{$\sqrt[n]{a^n+ b^n}$}, \mbox{$\sqrt[n]{c^n-a^n}$} and \mbox{$\sqrt[n]{c^n-b^n}$} are irrationals when $n$ is even greater than $2$. It is a step forward to prove that $c$ must be irrational to satisfy $a^n+b^n=c^n$ whenever $n$ is even greater than $2$.
%
%----------------------------------------------
% n EVEN, n>2
\subsubsection{d is odd greater than 1}
%----------------------------------------------
%
By the Fundamental Theorem of Arithmetic, $2^kd$ with $k\in\mathbb{Z}^+$ and $d$ odd represents all positive even integers. Clearly, if we apply the transitive axiom of equality among $2^h$, $\sqrt[n]{a^n+b^n}$ and $2^{k}d$, by Lemma~\ref{2knpar} (Case 1) follows that $2^{k}d=2^h\in\mathbb{Q}^c$. A similar situation occurs in cases 2 and 3. The next Lemma shows a formal proof, and the next Corollary covers all cases when $d\geq 1$.

\begin{Lema}
\label{2kdnpar}
Let $n$ be even greater than 2. If one element in the triple $(a,b,c)$ has the form $2^kd$ where $k\in\mathbb{Z}^+$ and $d$ is odd greater than 1, then there are no positive integers \mbox{$a,b,c$} satisfying the equation $a^n+b^n=c^n$.  In particular, Fermat's Last Theorem is valid.
\end{Lema}

\begin{proof}
Let us analyze the equation $c^n=a^n+b^n$ by cases.

\textit{Case 1.} Suppose \eqref{2kD=a+b} such that $d$ is an odd integer greater than $1$ and $k\in\mathbb{Z}^+$. Notice that $2^{k}d$ can be expressed as $2^{h}$, with $h\in\mathbb{R}$. Then $2^{k}d=2^{h}$. Using Lemma~\ref{2knpar} (Case 1) with $c=2^{h}$ and $a,b$ odd such that $[2^{h}]^{n}=a^{n}+b^{n}$, follows that  $2^h\in\mathbb{Q}^c$. Therefore, $2^{k}d=2^h\in\mathbb{Q}^c$. By contradiction, $2^{k}d\in\mathbb{Q}^{c}$. 

\textit{Case 2.} Assume \eqref{Ec2} such that $d$ is odd greater than $1$, and $k\in\mathbb{Z}^+$. Observe that $2^{k}d$ can be expressed as $2^{h}$, with $h\in\mathbb{R}$. Then $2^{k}d=2^{h}$. Invoking Lemma~\ref{2knpar} \mbox{(Case 2)} with $b=2^{h}$ and $a,c$ odd such that $c^{n}=a^{n}+[2^{h}]^n$, it is concluded that $2^{h}\in\mathbb{Q}^{c}$.  Therefore, by contradiction, $2^{k}d\in\mathbb{Q}^{c}$. 

\textit{Case 3.} 
Suppose \eqref{Ec3} with $d$ odd greater than $1$ and $n$ even greater than~$2$. Analogously, applying Lemma~\ref{2knpar} \mbox{(Case 3)} follows that in $c^{n}=[2^{h}]^n+b^n$ gives \mbox{$2^h=2^{k}d\in\mathbb{Q}^{c}$}.
\end{proof}
\subsubsection{Synthesis for n as an even integer greater than 2}
\begin{Coro}
\label{FLTpar}
Fermat's Last Theorem is true for all $n$ even integers greater than $2$.
\end{Coro}
\begin{proof}
It follows by combining Lemma~\ref{2knpar} and Lemma~\ref{2kdnpar}.
\end{proof}
The objective of this Section~\ref{npar} has been reached. Next, we will address the other half of the proof.
%
%----------------------------------------------
% n ODD, n>2
\subsection{n is an odd integer greater than 2}
\label{nimpar}
%----------------------------------------------
Our approach to this case is via mathematical artifices to make Corollary~\ref{trio2kd}'s equations look as if $n$ were even. That will enable us to use what we have already proved in Section~\ref{npar}. Again, let us address the problem by cases. First, we will consider triples $(a,b,c)$ that have perfect square numbers only, and then we will generalize it to any positive integer within the triple.
%
%----------------------------------------------
\subsubsection{a,b,c are perfect square numbers}
\label{squaresection}
%----------------------------------------------
%
\begin{Lema}
\label{square}
Let $n$ be odd, greater than 2. Then there are no perfect square numbers \mbox{$a,b,c$} that satisfy the equation $a^n+b^n=c^n$.  In particular, Fermat's Last Theorem is valid.
\end{Lema}

\begin{proof}
Consider \eqref{2kD=a+b}, \eqref{Ec2}, and \eqref{Ec3}, where $(a,b,c)$ is a triple with three perfect square numbers. By contradiction, we will prove that  \mbox{$2^{k}d \in \mathbb{Q}^{c}$.} 

\textit{Case 1}. Suppose \eqref{2kD=a+b} with $a,b$  odd integers, $2^{k}d$ an even integer, and all of them are perfect square numbers. Rewriting \eqref{2kD=a+b} similar to Pythagoras' theorem gives
\begin{align}
\label{pita}
[(2^{k}d)^{n/2}]^2=[a^{n/2}]^2+[b^{n/2}]^2.
\end{align}
\textit{Remark}. Since $a,b, 2^{k}d$ are positive and perfect square numbers, if we apply to each of them a square root, they are still integers. Recall that we are considering positive numbers only. Therefore, $a^{n/2}$, $b^{n/2}$ are odd, and $(2^{k}d)^{n/2}$ is even. It is evident that a \textit{necessary condition} for \eqref{pita} is that $(a^{n/2}, b^{n/2}, (2^{k}d)^{n/2})$ must be a \textit{primitive} Pythagorean triple. Let us consider a similar procedure as shown in Section~\ref{npar}.
\begin{enumerate}
    \item Let $d=1$. 
     This case is proved by following similar reasoning as shown in Lemma~\ref{2knpar} (Case 1) under the above-mentioned conditions. Therefore, by contradiction, $2^{k}\in\mathbb{Q}^c$.
    
    \item Let $d>1$ and $d$ odd. This case is proved by following similar steps as explained in Lemma~\ref{2kdnpar} (Case 1). Therefore,  by contradiction, we assert that $2^{k}d\in\mathbb{Q}^c$. 
    
    \item Synthesizing, $2^{k}d\in\mathbb{Q}^c$.
\end{enumerate}

\textit{Case 2}. Suppose \eqref{Ec2} where $a,c$  are odd and $2^{k}d$ is even, and all of them are perfect square numbers. Suppose $n$ is an odd integer greater than $2$. Let us follow an analogous procedure as shown in Case 1. Invoking the procedure of Lemma~\ref{2knpar} \mbox{(Case 2)} and Lemma~\ref{2kdnpar} (Case 2), we conclude that $2^{k}d\in\mathbb{Q}^c$.

\textit{Case 3}. Suppose \eqref{Ec3} where $b,c$  are odd and $2^{k}d$ is even, and all of them are perfect square numbers. Let us follow an analogous procedure as shown in \mbox{Case 1}. Invoking the procedure of Lemma~\ref{2knpar} (Case~3) and Lemma~\ref{2kdnpar} (Case~3), we conclude that $2^{k}d\in\mathbb{Q}^c$.
\end{proof}

\begin{Coro}
\label{square2}
Let $n$ be odd, greater than 2. Suppose  $(a,b,c)$ is a triple with two odd-perfect square numbers. Then the other element in the triple is irrational. In particular, Fermat's Last Theorem is valid.
\end{Coro}
%
%----------------------------------------------
\subsubsection{Generalization to a,b,c as positive integers}
\label{usarfig1}
%----------------------------------------------
%
Let us consider the arithmetic structure of equations that have the form \mbox{$c^n=a^n+b^n$} where the triple $(a,b,c)$ has two odd perfect square numbers and an irrational. For simplicity's sake, we are not interested in equations that have at least an even perfect square number and an irrational since they are not aligned to Corollary~\ref{square2}. However, one may see that the same conclusions arise in that \mbox{situation.}
\begin{Lema}
\label{abirrac}
Let $n$ be odd, greater than 2. Then there are no integers \mbox{$a,b,c$} that satisfies $a^n+b^n=c^n$. In particular, Fermat's Last Theorem is valid.
\end{Lema}
\begin{proof}
Let us present direct proof.

\textit{Case 1}. Consider the equation $[(2^kd)^2]^n=(a^2)^{n}+(b^2)^{n}$ where $(2^kd)^2, a^2$ and $b^2$ are perfect square numbers, $a,b,d,n$ are \textit{any} odd, and $n>2$. By Corollary~\ref{square2}, it follows that $(2^{k}d)^2\in\mathbb{Q}^c$. 

Considering $b\leq a<c$, and solving for $(b^2)^{n}$ gives
\begin{equation}
\left[\frac{(2^{k}d)^2}{b^2}\right]^{n}=\left[\frac{a^2}{b^2}\right]^{n}+1.
\label{frac2}
\end{equation}

Since $(2^{k}d)^2\in\mathbb{Q}^c$ and $b^2$ is odd, then $[(2^{k}d)^2/b^2]$ is irrational. $[(2^{k}d)^2/b^2]$ can be simply expressed as $2^{h}d$ such that $2^{h}d\in\mathbb{Q}^c$. In particular, let $2^h\in\mathbb{Q}^c$ and $d$ odd integer.  

Substituting $2^hd=[(2^{k}d)^2/b^2]$ in \eqref{frac2}, gives
\begin{equation}
(2^{h}d)^{n}=\left[\left(\frac{a}{b}\right)^2\right]^{n}+1, 
\label{frac}
\end{equation}
where $2^hd\in\mathbb{Q}^c$, $a.b, d, n$ are \textit{any }odd, $n>2$. Additionally, $a/b\geq 1$ and $(2^hd)^{n}>1$.

Next, we will use a mathematical artifice to transfer the existing irrationality on the left-hand side of the equal sign in \eqref{frac} to the right-hand side. 

Let $2^{h}d$ be a positive even integer: 
\begin{enumerate}
    \item Suppose $a=b$. By Proposition \ref{a=b}, it is impossible because it contradicts that $2^{h}d$ is an integer. 

    \item Assume $b\mid a$.  By Lemma~\ref{bnodividea}, it is impossible because it contradicts that $2^{h}d$ is an integer. 

    \item Suppose $b\nmid a$. By Corollary~\ref{noracionales} and the fact that $1=1^n$, it is impossible because it contradicts that $2^{h}d$ is an integer. 
\end{enumerate}
\textit{Remark.} If $2^{h}d\in\mathbb{Z}^+$, then a necessary condition for \eqref{frac} is that \mbox{$(a/b)^2\in\mathbb{Q}^c$}. Therefore,  $a/b\in\mathbb{Q}^c$. It implies that $a$ or $b\in\mathbb{Q}^c$ with $a\neq b$.

\textit{Case 2}. Consider the equation $(c^2)^{n}=(a^2)^{n}+[(2^{k}d)^2]^n$ where $c^2,a^2$ are perfect square numbers, $a,c,d,n$ are \textit{any} odd, $n>2$. Analogously to Case~1, we have that 
$[(c/a)^2]{n}=1+(2^{h}d)^{n}$, with $2^hd\in\mathbb{Q}^c$, such that $a,c$ are odd integers with $c/a> 1$. Let $2^{h}d$ be a positive even integer. Then, $a$ or $c\in\mathbb{Q}^c$.

\textit{Case 3}. Consider the equation 
$(c^2)^{n}=[(2^{k}d)^2]^n+(b^2)^{n}$ where $b^2,c^2$
are perfect square numbers, $b,c,d,n$ are \textit{any} odd, $n>2$. Analogously to Case~1, we have that 
$[(c/b)^2]{n}=(2^{h}d)^{n}+1$, where $2^hd\in\mathbb{Q}^c$, and $b, c$ odd integers with $c/b> 1$. Let $2^{h}d$ be a positive even integer. Then, \mbox{$b$ or $c\in\mathbb{Q}^c$}.
\end{proof}
%
%-----------------------------------------
\subsubsection{Synthesis for n as an odd integer greater than 2}
%----------------------------------------
\begin{Coro}
\label{FLTimpar}
Fermat's Last Theorem is valid for all $n$ odd integers greater than $2$.
\end{Coro}

%------------------------------------------------
\subsection{Synthesis for n as an integer greater than 2}
%------------------------------------------------
%
\begin{Teo}[Fermat's Last Theorem]
\label{FLT}
If $n$ is an integer greater than $2$, then there are no positive integers $a,b,c$ that satisfy the equation $a^n+b^n=c^n$.
\end{Teo}
\begin{proof}
It follows from considering Section~\ref{npar} and Section~\ref{nimpar}. It has been proved that a necessary condition for \mbox{$a^n+b^n=c^n$}, with $n$ integer greater than~$2$, is that within $(a,b,c)$ there is at least an irrational number. Synthesizing, we confirm that Fermat's Last Theorem is valid.  
\end{proof}

\begin{flushright}
\parbox[t]{10.5cm}{\lettrine{I}{t is impossible} \textit{ for any number which is a power greater than the second to be written as the sum of two like powers. Rei demonstrationem mirabilem sane detexi. Hanc marginis exiguitas non caperet.}}
\rightline{{\rm --- Pierre de Fermat}}
\end{flushright}
%
%
%------------------------------------------------
\section{Final considerations}
\label{final}
%------------------------------------------------
%
\subsection{Applications of Fermat's Last Theorem}
\label{app}

The Scottish mathematician \textit{John Napier} published his method of logarithms in 1614, just some years before Fermat's claim. Since we aim to use the mathematical knowledge of that \textit{époque}, we will apply logarithms to show some applications of Theorem~\ref{FLT}, also known as the \textit{Fermat-Wiles Theorem}.

\begin{Lema}
\label{2knpar2}
Suppose $a,b,c$ are positive integers that satisfy $a^n+b^n=c^n$. Then, $n$ is irrational.
\end{Lema}
\begin{proof}
Our first task is to consider the Cartesian analysis method, so we divide the problem into simple parts and then synthesize it.
\begin{enumerate}
 \item Suppose that $a,b,c$ are positive integers where one of them has the form $2^k$ with $k\in\mathbb{Z}^+$, and $gcd(a,b,c)=1$.

\textit{Case 1}. Assume a primitive Fermat's equation $a^n+b^n=c^n$ with $c=2^k$. Then, $2^{kn}=a^n+b^n$.  Solving for the exponent $kn$ gives
\begin{align}
\label{ksimple}
    kn=\log_2(a^n+b^n).
\end{align}
Consider the logarithmic argument in \eqref{ksimple}.  By the Fundamental Theorem of Arithmetic, we claim that $a^n+b^n=2^rs$ where $r,s\in\mathbb{Z}^+$ and~$s$ is an odd integer. By Theorem~\ref{FLT}, it follows that $n\notin \mathbb{Z}^+$ for $n>2$.  Then $n\in\mathbb{Q}^*$ or $n\in\mathbb{Q}^c$.

Rewriting \eqref{ksimple} gives
\begin{align}
\label{kn}
    kn=r+\log_{2}(s).
\end{align}

\begin{enumerate}
    \item Suppose $s>1$. Clearly$^\ddagger$, $\log_2(s)\in\mathbb{Q}^c$. Since $k,r\in\mathbb{Z}^+$, then a necessary condition for \eqref{kn} is that $n\in\mathbb{Q}^c$.

    \item The case when $s=1$ is absurd because otherwise \eqref{kn} contradicts Theorem~\ref{FLT}. Therefore, $n\in\mathbb{Q}^c$.
\end{enumerate}
Since $n\notin\mathbb{Q}^*$, then $n\in\mathbb{Q}^c$.

\textit{Case 2.} Suppose $b=2^k$. Let $c^n-a^n=2^rs$. Similar to Case 1, $n\in\mathbb{Q}^c$.

\textit{Case 3.} Consider $a=2^k$. Let $c^n-b^n=2^rs$. Therefore, $n\in\mathbb{Q}^c$.

\item Suppose that $a,b,c$ are positive integers where two of them are odd while the other has the form $2^kd$ with $k\in\mathbb{Z}^+$ and $d$ is odd greater than $1$. Additionally, $gcd(a,b,c)=1$. Similar to the previous item, solving for $kn$ gives
\begin{align}
\label{kn2}
kn=r+\log_{2}(s)-\log_{2}(d^n).    
\end{align}
Notice that Theorem~\ref{FLT} implies $s\neq d^n$. Moreover, the sum is well-defined, and it is irrational whenever $n$ is an integer or rational$^\ddagger$. It is evident that $n\notin\mathbb{Q}^*$. Since $k$ is an integer, a necessary condition for \eqref{kn2} is that $n\in\mathbb{Q}^c$. 
\end{enumerate}

Certainly, Item 1 can be deduced from Item 2 with $d=1$. Combining Item 1 and 2, it follows that, in general, $n\in\mathbb{Q}^c$. Figure~\vref{fig:Fig1} illustrates the range of $n$.
\end{proof}
%
%----------------------------------

\begin{Lema}
There are no rational numbers $a,b,c,n$ such that $a^n+b^n=c^n$. 
\end{Lema}
\begin{proof}
Let $a=p/q, \ b=t/u, \ c=v/w\in\mathbb{Q}^*$ and $n\in\mathbb{R}^+$. Then, $q,u$, and $w$ are positive integers greater than $1$. Let us construct a Fermat's equation and do some algebra with those numbers.
\begin{align*}
    \bigg(\frac{v}{w}\bigg)^n=&\left(\frac{p}{q}\right)^n+\left(\frac{t}{u}\right)^n,\\
    \left(\frac{v}{w}\right)^n=&\ \frac{(pu)^n+(qt)^n}{(qu)^n},\\
    \left(\frac{vqu}{w}\right)^n=&\ (pu)^n+(qt)^n,\\
    \left(vqu\right)^n=&\ (puw)^n+(tqw)^n.
\end{align*}
    Theorem~\ref{FLT} assures that $n$ is not an integer greater than $2$. Besides, Lemma~\ref{2knpar2} states that $n$ is strictly irrational. As a result, \mbox{$n\not\in\mathbb{Q}^*$}. (See the range of $n$ in Figure~\ref{fig:Fig1}).
\end{proof}

\begin{Lema}
The angle $\theta$ in Figure~\ref{fig:Fig1} is irrational. 
\end{Lema}
\begin{proof}
It follows from the cosine's law, $c^2=a^2+b^2-2ab\cos \theta$, and Theorem~\ref{FLT}.
\end{proof}

\subsection{Principles of analysis using complex numbers}
%---------------------------------------------
Notice that we may easily generalize Theorem~\ref{FLT} for $a,b,c,n$ when they are nonzero integers, and \mbox{$n\notin\{0, \pm 1,\pm 2\}$}. Now that we have a complete analysis of Fermat's equation within real numbers, one may ask how this equation behaves within complex numbers. It could probably be simpler or even more complex. We do not pretend to answer that question here, but at least we leave the first step toward that purpose.

Consider DeMoivre's Theorem, named after French mathematician \textit{Abraham de Moivre}. (1667-1754). Observe that beyond real numbers, $c(n)=\sqrt[n]{a^n+b^n}$ is not a one-to-one function since it has multiple non-real solutions.
%
%
%----------------------------------------------
\subsection{Rules of the Cartesian Method}
\label{metodo}
%----------------------------------------------
%
\begin{Rule}
\label{rule1}
\upshape 
Never to accept anything for true which I did not clearly know to be such; that is to say, carefully to avoid precipitancy and prejudice, and to comprise nothing more in my judgment than what was presented to my mind so clearly and distinctly as to exclude all grounds of doubt.
\end{Rule}

\begin{Rule}
\label{rule2}
\upshape 
To divide each of the difficulties under examination into as many parts as possible and as might be necessary for its adequate solution.
\end{Rule}

\begin{Rule}
\label{rule3}
\upshape 
To conduct my thoughts in such a way that, by commencing with objects the simplest and easiest to know, I might ascend by little and little, and, as it were, step by step, to the knowledge of the more complex; assigning in thought a certain order even to those objects which, in their own nature, do not stand in a relation of antecedence and sequence.
\end{Rule}

\begin{Rule}
\label{rule4}
\upshape 
To make enumerations so complete and reviews so general that I might be assured that nothing was omitted.

\begin{flushright}
\parbox[t]{9cm}{\textit{God alone is the author of all the motions in the world.}}
\rightline{{\rm --- René Descartes}}
\end{flushright}
\end{Rule}

%----------------------------------
\section*{Acknowledgements}
\addcontentsline{toc}{section}{Acknowledgements}
%----------------------------------
%\textit{Acknowledgements.}
The author is grateful to Danny Guerrero, \mbox{Ph.D.}, for his valuable support. danny.guerrero@unah.edu.hn. A special thanks to Jacques Gélinas, Ph.D. (Canada), for reviewing the manuscript and providing helpful suggestions. 
%
%------------------------------------------------
\section*{Notes}
\addcontentsline{toc}{section}{Notes}
%------------------------------------------------
%
\begin{tabularx}{\textwidth}{@{}>{\bfseries}l X@{}}
{$\dagger$} &  If $b\in \{1,2,3\}$, then there is no vertex $C$ in the arc $\beta \gamma$ of Figure~\ref{fig:Fig1}, such that $c\in \mathbb{Z}^+$ since  $\sqrt[n]{a^n+b^n}\in\mathbb{Q}^c$ for all $a\in\mathbb{Z}^+$ and $n>2$. The maximum number of vertices C that exist in the arc $\beta \gamma$ with $a,b,c\in \mathbb{Z}^+$ and $n\in \mathbb{R}^+$ is $h$, such that \mbox{$h\leq \lfloor a(\sqrt[3]{2}-1)\rfloor \lessapprox \lfloor 0$.$2599\ a\rfloor$}.\\
 & \\
{$\ddagger$} & By definition of logarithms, $\log_2(s)=x \Leftrightarrow 2^x=s$.  Recall that $s$ is odd and greater than $1$. Suppose that $x\in\mathbb{Q}^+$ such that $x=p/q$ and $q\neq 0$.  Then $2^p=s^q$ which is absurd, because the left-hand side of the equal sign is an even number while the right-hand side is an odd number.  Therefore, by contradiction, $\log_2(s)\in\mathbb{Q}^c$. (Similarly, $\log_{2}(d^n)\in\mathbb{Q}^c$).
\end{tabularx}
%
%------------------------------------------------
% REFERENCES
%------------------------------------------------
%

\footnotesize\noindent\scshape{Universidad Nacional Autónoma de Honduras}\\
\small\upshape\textit{E-mail:} hinunez@unah.hn

%Nunca habia saboreado la belleza y la amargura de las matemáticas.  Escritura iniciado el 10 de enero de 2021 y finalizado el 1 de abril de 2021. Revisiones y ajustes finalizaron el 19 mayo 2021. 
%Versión final 1: 20 junio de 2021.  
%Version final 2: 30 de junio 2021.  
%Version final 3: 3 de julio 2021. 
% Version 5:  Corrección de la traducción, y templates.
%(Comment: ESTO ES MAS LOGICA QUE MATEMATICAS)
%

%=================================
\end{document}